\documentclass[12pt]{amsart}
\usepackage{amssymb,amsmath,amscd}
\usepackage{mathrsfs}
\usepackage{color}
\newtheorem{thm}{\sc Theorem}[section]
\newtheorem{prop}[thm]{\sc Proposition}

\newtheorem{cor}[thm]{\sc Corollary}
\theoremstyle{definition}\newtheorem{exa}[thm]{\sc Example}
\theoremstyle{definition}\newtheorem{de}[thm]{\sc Definition}
\theoremstyle{definition}\newtheorem{rem}[thm]{\sc Remark}
\theoremstyle{definition}
\theoremstyle{definition}
\theoremstyle{definition}

\numberwithin{equation}{section}

\DeclareMathOperator{\C}{\mathbb C}

\ifx\chitalicit\undefined\let\chitalicit\relax\fi
\ifx\chbfit\undefined\let\chbfit\relax\fi
\usepackage{pifont}
\begin{document}
\title[Nonnegativity Criterion for Paneitz operators]{Embedded three-dimensional CR manifolds and the Non-negativity of Paneitz Operators}
\author[S.Chanillo, H.-L. Chiu, and P. Yang]{SAGUN CHANILLO, HUNG-LIN CHIU AND PAUL YANG}
\address{S. Chanillo, Department of Mathematics, Rutgers University, 110 Frelinghuysen Rd., Piscataway, NJ 08854, U.S.A.}
\email{chanillo@math.rutgers.edu}

\address{H.-L. Chiu, Department of Mathematics, National Central University, Chung Li, 32054, Taiwan, R.O.C.}
\email{hlchiu@math.ncu.edu.tw}

\address{P. Yang, Department of Mathematics, Princeton University, Princeton, NJ 08544, U.S.A.}
\email{Yang@math.princeton.edu}
\keywords{}\subjclass{Primary 32V05, 32V20; Secondary 53C56.
{\color{red} v19}}
\renewcommand{\subjclassname}{\textup{2000} Mathematics Subject
Classification}

\begin{abstract}
Let $\Omega\subset\mathbb{C}^{2}$ be a strictly pseudoconvex domain
and $M=\partial\Omega$ be a smooth, compact and connected CR
manifold embedded in $\mathbb{C}^2$ with the CR structure induced
from $\mathbb{C}^{2}$. The main result proved here is as follows.
Assume the CR structure of $M$ has zero torsion. Then if we make a
small real-analytic deformation of the CR structure of $M$ along
embeddable directions, the CR structures along the deformation path
continue to have non-negative Paneitz operators. We also show that
any ellipsoid in $\mathbb{C}^2$ has positive Webster curvature.
\end{abstract}

\maketitle
\section{introduction}

Throughout this paper, we will use the notation and terminology in
(\cite{Le}) unless otherwise specified. Let $(M,J,\theta)$ be a
smooth, closed and connected three-dimensional pseudo-hermitian
manifold, where $\theta$ is a contact form and $J$ is a CR structure
compatible with the contact bundle $\xi=\ker\theta$. The CR
structure $J$ decomposes $\bf \C\otimes\xi$ into the direct sum of
$T_{1,0}$ and $T_{0,1}$ which are eigenspaces of $J$ with respect to
$i$ and $-i$, respectively. The Levi form $\left\langle\  ,\
\right\rangle_{L_\theta}$ is the Hermitian form on $T_{1,0}$ defined
by $\left\langle Z,W\right\rangle_{L_\theta}=-i\left\langle
d\theta,Z\wedge\overline{W}\right\rangle$. We can extend
$\left\langle\  ,\ \right\rangle_{L_\theta}$ to $T_{0,1}$ by
defining $\left\langle\overline{Z}
,\overline{W}\right\rangle_{L_\theta}=\overline{\left\langle
Z,W\right\rangle}_{L_\theta}$ for all $Z,W\in T_{1,0}$. The Levi
form induces a natural Hermitian form on the dual bundle of
$T_{1,0}$, denoted by $\left\langle\  ,\
\right\rangle_{L_\theta^*}$, and hence on all the induced tensor
bundles. Integrating the hermitian form (when acting on sections)
over $M$ with respect to the volume form $dV=\theta\wedge d\theta$,
we get an inner product on the space of sections of each tensor
bundle. We denote the inner product by the notation $\left\langle \
,\ \right\rangle$. For example
\begin{equation}\label{21}
\left\langle\varphi ,\psi\right\rangle=\int_{M}\varphi\bar{\psi}\ dV,
\end{equation}
for functions $\varphi$ and $\psi$.

Let $\left\{T,Z_1,Z_{\bar{1}}\right\}$ be a frame of $TM\otimes \bf C$, where $Z_1$ is any local frame of $T_{1,0},\  Z_{\bar{1}}=\overline{Z_1}\in T_{0,1}$
and $T$ is the characteristic vector field, that is, the unique vector field such that $\theta(T)=1,\ d\theta(T,\cdot)=0$.
Then $\left\{\theta,\theta^1,\theta^{\bar 1}\right\}$, the coframe dual to $\left\{T,Z_1,Z_{\bar{1}}\right\}$, satisfies
\begin{equation}\label{22}
d\theta=ih_{1\bar 1}\theta^1\wedge\theta^{\bar 1}
\end{equation}
for some positive function $h_{1\bar 1}$. We can always choose $Z_1$
such that $h_{1\bar 1}=1$; hence, throughout this paper, we assume
$h_{1\bar 1}=1$

  The pseudohermitian connection of $(J,\theta)$ is the connection
$\nabla$ on $TM\otimes \bf C$ (and extended to tensors) given in terms of a local frame $Z_1\in T_{1,0}$ by

\begin{equation*}
\nabla Z_1=\theta_1{}^1\otimes Z_1,\quad
\nabla Z_{\bar{1}}=\theta_{\bar{1}}{}^{\bar{1}}\otimes Z_{\bar{1}},\quad
\nabla T=0,
\end{equation*}

where $\theta_1{}^1$ is the $1$-form uniquely determined by the following equations:

\begin{equation}\label{id10}
\begin{split}
d\theta^1&=\theta^1\wedge\theta_1{}^1+\theta\wedge\tau^1\\
\tau^1&\equiv 0\mod\theta^{\bar 1}\\
0&=\theta_1{}^1+\theta_{\bar{1}}{}^{\bar 1},
\end{split}
\end{equation}
where $\theta_{1}{}^{1}$ and $\tau^1$ are called the connection form
and the pseudohermitian torsion, respectively. Set
$\tau^1=A^1{}_{\bar 1}\theta^{\bar 1}$. The structural equation for
the pseudohermitian connection is given by,

\begin{equation}\label{strueq}
d\theta_{1}{}^{1}=Rh_{1\bar 1}\theta^{1}\wedge\theta^{\bar 1}+A_{1}{}^{\bar 1}{}_{,\bar 1}\theta^{1}\wedge\theta
-A_{\bar 1}{}^{1}{}_{,1}\theta^{\bar 1}\wedge\theta.
\end{equation}

where $R$ is the Tanaka-Webster curvature, see \cite{W}.

We will denote components of covariant derivatives with indices preceded by a comma; thus we write $A^{\bar{1}}{}_{1,\bar{1}}\theta^1\wedge\theta$.
The indices $\{0, 1, \bar{1}\}$ indicate derivatives with respect to $\{T, Z_1, Z_{\bar{1}}\}$.
For derivatives of a scalar function, we will often omit the comma,
for instance, $\varphi_{1}=Z_1\varphi,\  \varphi_{1\bar{1}}=Z_{\bar{1}}Z_1\varphi-\theta_1^1(Z_{\bar{1}})Z_1\varphi,\  \varphi_{0}=T\varphi$ for a (smooth) function.

Next we introduce several natural differential operators occuring in
this paper. For a detailed description, we refer the reader to the
article \cite{Le}. For a smooth function $\varphi$, the
Cauchy-Riemann operator $\partial_{b}$ can be defined locally by
\[\partial_{b}\varphi=\varphi_{1}\theta^{1},\]
and we write $\bar{\partial}_{b}$ for the conjugate of
$\partial_{b}$. A function $\varphi$ is called CR holomorphic if
$\bar{\partial}_{b}\varphi=0$. The divergence operator $\delta_b$
takes $(1,0)$-forms to functions by
$\delta_b(\sigma_{1}\theta^{1})=\sigma_{1,}{}^{1}$, and similarly,
$\bar{\delta}_b(\sigma_{\bar 1}\theta^{\bar 1})=\sigma_{\bar
1,}{}^{\bar 1}$.

If $\sigma =\sigma_{1}\theta^{1} $ is compactly supported, Stokes' theorem applied to the 2-form $\theta\wedge\sigma $ implies the divergence formula:
\[\int_{M}\delta_{b}\sigma \theta\wedge d\theta=0.\]
It follows that the formal adjoint of $\partial_{b}$ on functions
with respect to the Levi form and the volume element $\theta\wedge
d\theta$ is $\partial_{b}^{*}=-\delta_{b}$. The Kohn Laplacian on
functions is given by the expression,
\[\Box_{b}=2\bar{\partial}_{b}^{*}\bar{\partial_{b}}.\]

Define
\begin{equation}\label{opP3}P_3\varphi=(\varphi_{\bar{1}}{}^{\bar{1}}{}_1+iA_{11}\varphi^1)\theta^1\end{equation}
(see \cite{Le}) which is an operator whose vanishing  characterizes
CR-pluriharmonic functions.

We also define.
$\overline{P}_3\varphi=(\varphi_{1}{}^{1}{}_{\bar{1}}-iA_{\bar{1}\bar{1}}\varphi^{\bar{1}})\theta^{\bar{1}}$,
the conjugate of $P_3$.

\begin{de}The CR Paneitz operator $P_4$ is defined by
\begin{equation}\label{opP0}P_4\varphi=\delta_b(P_3\varphi).\end{equation}\end{de}
More explicitly,
define $Q$ by $Q\varphi=2i(A^{11}\varphi_{1})_{,1}$, then
\begin{equation*}
\begin{split}
P_{4}\varphi&=\frac{1}{4}(\Box_{b}\overline{\Box}_{b}-2Q)\varphi\\
&=\frac{1}{4}(\Box_{b}\overline{\Box}_{b}\varphi-4i(A^{11}\varphi_{1})_{1})\\
&=\frac{1}{8}\big((\overline{\Box}_{b}\Box_{b}+\Box_{b}\overline{\Box}_{b})\varphi+8Im(A^{11}\varphi_{1})_{1}\big).
\end{split}
\end{equation*}
By the commutation relation $[\Box_{b},\overline{\Box}_{b}]=4iImQ$,
we see that
$4P_{4}=\Box_{b}\overline{\Box}_{b}-2Q=\overline{\Box}_{b}\Box_{b}-2\overline{Q}$.
It follows that $P_{4}$ is a real and symmetric operator (see
\cite{C} for details).

\begin{de}
We say the Paneitz operator $P_{4}$ is nonnegative if and only if
\[\int_{M}(P_{4}\varphi)\bar{\varphi}\geq 0,\]
for all smooth functions $\varphi$. We use the notation $P_{4}\geq
0$ to denote non-negative Paneitz operators.
\end{de}
Note that the nonnegativity of $P_{4}$ is a CR invariant in the
sense that it is independent of the choice of the contact form
$\theta$. This follows by observing that if
$\widetilde{\theta}=e^{2f}\theta$ be another contact form, we have
the following transformation laws for the volume form and the CR
Paneitz operator respectively (see Lemma 7.4 in \cite{H}):
\[\widetilde{\theta}\wedge d\widetilde{\theta}=e^{4f}\theta\wedge d\theta;\ \ \ \ \ \  \widetilde{P_{4}}=e^{-4f}P_{4}.\]

In the higher dimensional case, there exists an analog of $P_{4}$
which satisfies the covariant property. In this case, Graham and
Lee, in \cite{GL}, had shown the nonnegativity of $P_{4}$. To be
specific, non-negativity of $P_4$ is a condition in dimension three
but it is a given in higher dimensions. Moreover the invariance
property for the Paneitz discussed above does not hold in dimensions
five and higher.

We will restrict ourselves exclusively to the three dimensional case
in our paper. We next observe that when the Webster torsion
$A_{11}\equiv 0$, then the Paneitz operator $P_{4}$ is given by,

\begin{equation}\label{zerotor}P_{4}=\frac{1}{4}\Box_{b}\overline{\Box}_{b}.\end{equation}

It follows that the vanishing of torsion implies that $P_{4}\geq 0$.
This is because when the torsion vanishes identically, the two
operators $\Box_{b}$ and $\overline{\Box}_{b}$ commute, and hence
are simultaneously diagonalizable on each eigenspace of $\Box_{b}$
of a nonzero eigenvalue(see \cite{C}). We also recall that the
vanishing of torsion is equivalent to $L_{T}J=0$ where $L$ is the
Lie derivative, see \cite{W}. We summarize a part of the facts above
as a proposition, which will prove useful later.
\begin{prop}\label{pospan}Let the Webster torsion tensor identically vanish, i.e. $A^{11}\equiv 0$. Then,
 \begin{equation}\label{kernelpan} \ker P_4=\ker P_3=\mbox{CR-pluriharmonic
 functions.}\end{equation}
   \begin{equation}\label{pospantor}P_4\geq 0.\end{equation}
 \end{prop}

It remains an interesting problem to determine the precise
geometrical condition under which the kernel of the Paneitz operator
is exactly the pluri-harmonic functions or even a direct sum of a
finite dimensional subspace with the pluri-harmonic functions.

\begin{de}
Suppose that $\widetilde{\theta}=e^{2f}\theta$. The CR Yamabe constant is defined by\\
$\inf_{\widetilde{\theta}}{\{\int_{M}\widetilde{R}\ \widetilde{\theta}\wedge d\widetilde{\theta} : \int\widetilde{\theta}\wedge d\widetilde{\theta}=1\}}$.
\end{de}
The CR Yamabe constant is a CR invariant.

We now come to the primary results of our paper. To motivate the
results, it is helpful to recall the main result in our earlier
paper \cite{CCY}.
\begin{thm}\label{dukethm}
Let $M^{3}$ be a closed CR manifold.

(a) If $P_{4}\geq 0$ and $R>0$, then the non-zero eigenvalues
$\lambda$ of $\Box_{b}$ satisfy $$\lambda\geq\min{R}.$$

It follows the range of $\Box_{b}$ is closed. Coupled with the
result of Kohn stated above, under the conditions $P_{4}\geq 0$ and
$R>0$, $M$ globally embeds into some $\C^{n}$.

(b) A consequence of part (a) is that: If $P_{4}\geq 0$ and the CR
Yamabe constant $>0$, then $M^{3}$ can be globally embedded into
$\C^{n}$, for some $n$.
\end{thm}

Our aim is to investigate a converse to the theorem stated above.
More specifically we want to know if for embedded structures, the CR
Paneitz operator is non-negative. We recall the following example
due to Grauert, Andreotti-Siu \cite{AS} and Rossi \cite{R} and
referred to as Rossi's example in the literature \cite{CS}.

\begin{exa}On the standard sphere $(S^{3},J_{0})$, we consider the deformation
$J_{t}$ given by the vector field $Z_{\bar{1}}+tZ_1$, with $t\in
\mathbb{R}$ and $|t|\not=0,1$. This structure fails to embed
globally since it is known that the CR functions for this structure
are even. We note $u_{t}=z_{1}$ (which is an odd function) is a
continuous family of eigenfunctions for $P_{4}^{t}$ with eigenvalue
$\lambda(t)=\frac{-3t^2}{(1-t^2)^2}$. This means that we are unable
to find a CR function $\varphi_{t}$ for the CR structure
$(S^{3},J_{t})$ which is as close to $u_{0}=z_{1}$ as we please. The
key observation is that the Paneitz operator is negative and the
structure fails to embed.
\end{exa}

The example thus suggests that indeed it is possible that for
embedded structures the CR Paneitz operator may indeed be
non-negative. The main result in Section \ref{locdef} of our paper
is a result that ensures non-negativity of the Paneitz operator, for
CR structures embedded in $\mathbb{C}^2$ along a deformation path
that is real-analytic. More precisely, we are given a triple $(M,
J_0,\theta)$, the background CR structure. This CR structure is
given to be embedded in $\mathbb{C}^2$. Now we deform the almost
complex structure $J_0$ via a real-analytic path $J_t$, keeping off
course the contact form $\theta$ fixed. That is each CR structure
along the path of deformation $J_t$ is smooth for fixed $t$, but the
dependence is real-analytic in the variable $t$. In the sequel when
we perform deformations, the Paneitz operator associated to the
deformed structures $J_t$ will be denoted by $P_4^t$. The Paneitz
operator for the reference structure $J_0$ will be denoted by $P_4$
instead of $P_4^0$.

\begin{thm}\label{main}
 Let $(M,J_0,\theta)$ be a CR structure that is embedded in $\mathbb{C}^2$.
 Let $J_{t}$ be a deformation from $J_0$ along an embeddable
direction with real-analytic dependence on the deformation parameter
$t$. Assume, each structure $J_t$ for fixed $t$ is smooth and
embedded in $\mathbb{C}^2$. Let $P_{4}^t$ denote the Paneitz
operator associated to the structure $(M,J_{t},\theta)$. Assume
further that the Paneitz operator at $t=0$,  $P_4( =P_4^0)$ is
non-negative and
           $$ \ker P_4=\mbox{CR-pluriharmonic functions.}$$
 Then for some $\delta>0$ and $|t|<\delta$
we have:

   $$P_{4}^{t}\geq 0.$$

\end{thm}
\begin{cor} Under the hypothesis of zero torsion, $A^{11}\equiv 0$,
the hypothesis of Theorem \ref{main} are met by virtue of
Proposition \ref{pospan}. Thus if the structure $J_0$ has zero
torsion, it follows $P_4^t\geq 0$ for $|t|<\delta$.\end{cor}
\begin{rem}
Theorem \ref{main} is a consequence of a local deformation theorem
proved in Section \ref{locdef}. It is based in part on the stability
of CR functions and a theorem of Lempert \cite{Lem}. It is also
important to note that in light of Rossi's example, the hypothesis
that $J_t$ is an embedded structure along the deformation path,
cannot be removed.
\end{rem}

\begin{rem} In the theorem above, we need to start deforming from a
manifold which is embedded and whose CR Paneitz operator is
non-negative. Examples of such manifolds are many. The sphere $S^3$
is such a manifold. The CR structure remains invariant under a
circle action and as remarked above, this forces the CR structure to
have vanishing torsion and so as observed above, the CR Paneitz
operator for the sphere is non-negative.

The sphere is simply-connected. We can consider now the manifold for
$(z,w)\in \mathbb{C}^2$ given by
                  $$ |z|^2+\frac{1}{|z|^2}+|w|^2=100.$$
It is evident that the CR structure is invariant under a circle
action. It is also evident that the manifold is not simply
connected. Thus this example provides an example of a starting
structure that is not simply connected and has a non-negative
Paneitz operator.
\end{rem}
\begin{rem}A result in \cite{CC}, Prop. 4.1 states that for embedded
structures $M$:
\begin{equation}\label{paneq}
      c\int_M |f|^2\leq \int_M|P_4f|^2
\end{equation}
which is valid $\forall$ $f\in(\ker P_4^t)^\perp$, with $c>0$ and
independent of $f$. That is $P_4$ has closed range for embedded
structures. However it is not obvious that when one performs a
deformation along embedded directions, the constant $c$ in the
inequality above stays uniformly positive. If one were to obtain a
uniform positive lower bound for $c$ along the deformation path, one
would be able to improve the conclusion of Theorem \ref{main} to a
global result valid for all $t$ in any compact interval containing
$t=0$\end{rem}

This brings us to the remaining part of the converse in Theorem
\ref{dukethm}. That is, do embedded structures have positive Yamabe
constant or positive Webster curvature. This is unlikely globally
but certainly true if the CR structures are small perturbations of
the standard CR structure of $S^3$. This is just by continuity. In
fact by continuity if one performs a small perturbation from any CR
structure whose Webster curvature is positive, the deformed CR
structure does have positive Webster curvature. However, for one
large class of important hypersurfaces in $\mathbb{C}^2$, the
ellipsoids, we do show that the Webster curvature is positive no
matter how much deformed the ellipsoid is. The principal result in
Section \ref{Webcur} is:
\begin{thm}\label{main2}
The Webster curvature for all ellipsoids is positive.
\end{thm}
We have been informed by Song-Ying Li that he too was aware of the
theorem stated above.

Now we specialize the situation to $S^3$ and consider small
deformations of the standard CR structure of the sphere. In
particular our goal is to consider the deformed structure on $S^3$
given by,
\[Z_{\bar{1}}^{t}=Z_{\bar{1}}^{\phi_t} = F(Z_{\bar{1}}+t\phi Z_{1}),\]
where $F=(1-t^{2}|\phi|^{2})^{-1/2}$, $Z_{\bar{1}}= \bar{z}
_{2}\frac{\partial}{\partial
Z_{1}}-\bar{z}_{1}\frac{\partial}{\partial Z_{2}}$ and $t\in
(-\epsilon,\epsilon )$. The factor $F$ is introduced to normalize
the Levi form so that $h_{1\bar{1}}\equiv 1$. The CR Paneitz
operator for the deformed structure will be denoted by $P_{0}^{t}$.
We now consider the 3-sphere $S^{3}\subset\C^{2}\ni (z_{1},z_{2})$
and denote by
\[P_{p,q}=\textrm{span} \{ z_{1}^{a} z_{2}^{b} \bar{z}_{1}^{c} \bar{z}_{2}^{d}| a+b=p,\ c+d=q\} \]
and the spherical harmonics
\[H_{p,q}=\{f \in P_{p,q} |-\Delta_{s^{3}}f=(p+q)(p+q+2)f\}.\]
 For a given $\phi\in C^{\infty }(S^{3})$ one has the Fourier representation
\[\phi \sim \sum \phi _{pq}\]
where $\phi _{pq}$ is the projection of $\phi $ onto $H_{p,q}$.

\begin{de}
We say $\phi $ satisfies condition (BE) if and only if
\[\phi _{pq}\equiv 0\  \textrm{for}\ \  p<q+4,\ q=0,1,\cdots.\]
\end{de}

\begin{rem}
Since for $p>q$
\begin{equation*}
P_{p,q}=H_{p,q}\oplus\cdots\oplus H_{p-q,0}.
\end{equation*}
It follows that if $\phi\in P_{p,q}$, then $\phi$ satisfies (BE) if
and only if $p\geq q+4$. Furthermore, the example of Rossi
corresponds to $\phi=1$ and thus fails condition (BE).
\end{rem}

Burns and Epstein proved in \cite{BC} that for $t\in (-\epsilon
,\epsilon )$ and $\phi $ satisfying (BE) the CR structure embeds
into some $\C^{n}$. Conversely Bland \cite{B} showed that
embeddability of a CR structure close to the standard structure on
$S^{3}$
implies condition (BE). To summarize we have \\

{\bf Theorem }[\ {\bf Burns-Epstein-Bland}]. A CR structure close to
the standard structure on $S^3$ is embeddable if and only if $\phi$
satisfies condition (BE).\\

One of the results proved in Section \ref{locdef} our paper, which
is obtained by combining the results in our earlier paper
\cite{CCY}, Theorem \ref{dukethm} with the results obtained in
Section \ref{locdef} of this paper and the theorem of
Burns-Epstein-Bland cited above is:
\begin{thm}\label{deformationthm}
Let us consider the three sphere $S^3$ and a CR structure $J_t$
obtained as a small perturbation of the standard CR structure on
$S^3$ and whose CR vector field is given by $Z_{\bar 1}^t$ above.
Then the following are equivalent.
\begin{enumerate}
\item The CR structure embeds in $\mathbb{C}^2$.
\item $\Box_b^t$, the Kohn Laplacian for the deformed structure has
closed range.
\item The deformation function $\phi(\cdot)$ used to define the CR
vector field $Z_{\bar 1}^t$, satisfies the Burns-Epstein condition
(BE).
\item The CR Paneitz operator $P_4^t$ for the deformed structure is
non-negative and the Yamabe constant for the deformed structure is
positive.
\end{enumerate}\end{thm}

As pointed out earlier, the Yamabe constant is positive for the
deformed structure and follows simply by continuity and the fact we
are only making a small deformation of the standard structure on
$S^3$. The Yamabe constant is of course positive for the standard CR
structure on $S^3$.

{\bf Acknowledgment.} The first author's research was supported in
part by NSF grant DMS-0855541, the second author's research was
supported in part by CIZE Foundation
 and in part by NSC 96-2115-M-008-017-MY3, and the third author's research was supported in part by
NSF grant DMS-0758601. We also thank C. Epstein for providing the
proof of Lemma (\ref{epslemma}). S.C. wishes to thank E. Bedford for
a helpful conversation.

\section{Small deformations of a CR structure}\label{locdef}
In the sequel we will always assume  $D\subset\C^{2}$ is a strictly
pseudoconvex, bounded domain with $(M,J_0)=\partial D$, in
particular $M$ is compact. Suppose that $J_t$ be a deformation from
$J_{0}$ defined by a family of smooth functions in the coordinate
variable of the manifold denoted by $\cdot$ and real analytic in the
deformation parameter variable $t$. The deformation functions on $M$
will be denoted by
 $\psi(\cdot,t)$. That is, the vector field
$\overline{Z}_{1}^{t}=\overline{Z}_{1}+\psi(\cdot,t)Z_{1}$ defines a
CR holomorphic vector field with respect to $J_{t}$. We also fix
notation and denote the CR Paneitz operator wrt to the background CR
structure $J_0$ as $P_4$ instead of $P_4^0$.

We now define the notion of stability for the Paneitz operator.
\begin{de} We say the Paneitz operator $P_4^{t_0}$ associated to the CR structure $J_{t_0}$ is stable,
if given $\epsilon>0$, there exists $\delta>0$ such that for all $t$
such that $|t-t_0|<\delta$, and given any $f\in \ker P_4^{t_0}$,
there exists $g\in \ker P_4^{t}$ such that
               $$||f-g||_{C^0(M)}<\epsilon.$$
               \end{de}
There is a similar notion for the stability of CR functions.
Stability of CR functions was established in a paper by Lempert
\cite{Lem}.

The proof of the next proposition was communicated to us by C.
Epstein \cite{CE}. For our purposes we need the projection operators
constructed in Prop. (8.18) in \cite{E}, except for the zero
eigenspace, to be continuous even at $t=0$. This is the content of
the following proposition. To state the lemma we need a few facts.
We consider a family $L_t$ of operators on $M$, that is holomorphic
in $t\in \mathbb{C}$ for $|t|<\delta$. For real $t$ we assume that
the operators $L_t$ are Hermitian with respect to $L^2(M)$ defined
using a fixed measure independent of $t$ which for our purposes is
$\theta\wedge d\theta$. Our operators $L_t$ are densely defined on
$C^\infty(M)$ and the examples we need them for are Kohn's Laplacian
$\Box_b^t$ and $P_0^t$. We assume moreover that
\begin{enumerate}
\item Each $L_t$ has closed range.
\item Each $L_t$ has pure point, discrete eigenvalues with finite dimensional
eigenspaces.
\item In particular it follows from the above two assumptions that for each $L_t$ we do not have
non-zero eigenvalues with zero as limit point.
\item We assume the spectrum of $L_t$ is bounded below.
\end{enumerate}
Since we will apply the proposition to families of Paneitz operators
$P_0^t$ associated to embedded families of CR structures $(M,J_t)$
and associated Kohn Laplacians $\Box_b^t$, we note that the closed
range hypothesis for embedded structures is satisfied for $P_0^t$ by
a result in \cite{CC} and for $\Box_b^t$ by a result in \cite{Ko}.
Now further assume there exists $r>0$, such that
      \[ ((-r,0)\cup(0,r))\cap\mbox{spectrum}\ L_0=\emptyset .\]
Non-zero eigenvalues of $L_t$ that lie in $(-r,r)$ will be called
small, using the terminology of \cite{E}.

\begin{prop}\label{epslemma} Let $L_z$ be a holomorphic family as above. Then
the small eigenvalues of $L_t$ are finitely many and depend
real-analytically on $t$ for $t\in (-\delta,\delta)$. The projection
$\mathbb{P}_i^t$ into the eigenspace for the small eigenvalue
$\lambda_i(t)$ of $L_t$ depends real-analytically on $t\in
(-\delta,\delta)$. Moreover if $\mathbb{P}^t$ denotes the projection
into the small eigenvalues, then the rank of $\mathbb{P}^t$ is
constant in $t$.\end{prop}
\begin{proof} In Proposition (8.18) \cite{E}, the analytic
dependence of the small eigenvalues is already established. What
remains to be proven is the second part of our lemma. Recall the
definition of $P_i^t$ eqn. (8.23) in \cite{E} which is,
       \begin{equation}\label{prop818} P_i^t=\lambda_i(t)\Pi_{j\neq
       i}(\lambda_i(t)-\lambda_j(t))\mathbb{P}_i^t.\end{equation}
       For $u,v\in L^2(M)$, define the function $g(t)$
          \[ g(t)=\frac{<P_i^tu,v>}{\lambda_i(t)\Pi_{j\neq i}(\lambda_i(t)-\lambda_j(t))}.\]
Then $g(t)$ is holomorphic in a punctured nbhd. of $t=0$. The
function $g(t)$ can have only poles of finite order as singularities
at $t=0$ and on the real axis via (\ref{prop818}), for
$|t|<\varepsilon$ the function $g(t)$ is bounded. Thus the
singularity at $t=0$ is removable and then arguing now as the rest
of Proposition (8.18) in \cite{E} we conclude that the projection
operators are real-analytic and converges to a finite rank
projection operator at t=0. Since
   $${rank}\ \mathbb{P}^t= {trace}\ \mathbb{P}^t,$$
   we obtain the integer valued function ${rank}\ \mathbb{P}^t$ is continuous and hence constant.\end{proof}

\begin{prop}\label{keyproppan}
Suppose $(M,J_0)$ is embedded in $\mathbb{C}^n$. Let $J_{t}$ be a
deformation from $J_0$ along an embeddable direction, with $t$
varying real-analytically and $|t|<\delta$. Let $P_4\geq 0$ and
assume further that the CR Paneitz operator $P_4$ for the structure
$J_0$ is stable. Then $P_4^t$ cannot have small eigenvalues. In
particular there does not exist any continuous family of
eigenfunctions $u_{t}$ corresponding to non-zero eigenvalues of
$P_4^{t}$ branching out from a function $u_0$ in the kernel of
$P_4$.  One therefore concludes $P_4^t\geq 0$.
\end{prop}
\begin{proof} We argue by contradiction. Assume $P_4^t$ has small
eigenvalues. Then by Prop.\ref{epslemma}, the eigenvalues vary
continuously in $t$ and the projection operators to these non-zero
eigenvalues $\mathbb{P}_i^t$ are also continuous. From the
continuous dependence of $\mathbb{P}_i^t$ and $\lambda_i(t)$ on $t$
we conclude that any eigenfunction $u_t$ for a non-zero small
eigenvalue can be written as $u_t= u_0+f_t$, where $u_0$ is in the
kernel of $P_4$, and $||f_t||_2=o(1)$. We normalize $||u_0||_2=1$.
From our stability assumption, there exists a function $g_t$ in
$\ker P_4^t$ such that $\|u_0-g_t\|<\epsilon$. Now $<u_t,g_t>=0$ as
they are eigenfunctions for distinct eigenvalues of $P_4^t$. Thus
for each $t\neq 0$ small enough we have

\begin{equation}
\begin{split}
0&=<u_{t},g_{t}>\\
&=<u_t,u_0>+<u_t,g_t-u_0>\\
&=1+<f_t,u_0>+<u_t,g_t-u_0>=1+o(1)\neq 0,
\end{split}
\end{equation}
which is a contradiction. Thus there are no small eigenvalues of
$P_4^t$. The operator $P_4\geq 0$ by assumption. Thus it follows
that $P_4^t\geq 0$ for $|t|<\delta$.
\end{proof}

We are now in a position to supply the proof of Theorem \ref{main}.
It is a consequence of the next proposition.
\begin{prop}\label{pluriext} Assume the kernel of $P_4$ consists of exactly
the CR pluri-harmonic functions for the structure $(M,J_0)$. Assume
the CR structures $(M,J_t)$ are all embedded in $\mathbb{C}^2$. Then
the Paneitz operator $P_4$ associated to the structure $J_0$ is
stable.\end{prop}

\begin{proof} By assumption any function $f\in\ker P_4$ is a CR pluriharmonic function. Locally
then $f$ is the real part of a CR holomorphic function $F$. We may
now locally extend $F$ into $\Omega$ where $M=\partial\Omega$. We
continue to denote the extension by the symbol $F$. We now denote
points in $\mathbb{C}^2$ by $(z,w)$. Next note in $\Omega$ that
$(Re\ F)_z$ is a holomorphic function defined globally in a  nbhd of
$M$ in $\Omega$. This is because $Re\ F=f$ is globally defined on
$M$. Since $M$ is connected, by Hartog's theorem we can even assume
that $(Re \ F)_z$ is defined in all of $\Omega$. Let us denote the
restriction to $M$ of $(Re\ F)_z$ by $\Xi$. Now $\Xi$ is a CR
function. We apply the stability theorem of Lempert \cite{Lem} to
obtain a function $\Xi_t$ which is a CR function for the structure
$J_t$ and such that
                   $$||\Xi-\Xi_t||_\infty<\epsilon.$$
Being a CR function $\Xi_t$ lies in the kernel of $P_4^t$. Next we
consider the extension of $\Xi_t$ to the interior as a holomorphic
function. This globally exists by Hartog's theorem again. We
continue to denote this extension by $\Xi_t$. Next we integrate
$\Xi_t$ in the $z$ variable, that is we consider the indefinite
integral
            $$F_t(z,w)=\int \Xi_t\ dz.$$
There may be an ambiguity in the definition of $F_t$, because of
imaginary periods but the Real part of $F_t$ is well-defined. Set
$f_t=Re\ F_t$. Then $f_t$ is pluriharmonic and its restriction to
$M$ is CR-pluriharmonic. Similarly we also consider
           $$H(z,w)=\int \Xi\ dz$$
Note that $H(z,w)$ may differ from $F$ because of imaginary periods.
But their real parts do coincide.

We now easily see using the the stability estimate above,
                        $$||f-f_t||_{L^\infty(M)}<\epsilon.$$
We have proved stability.

If $M$ were simply connected then the proof of the proposition is
quite easy, since then $f$ being CR pluriharmonic can be taken to be
the real part of a CR function $G$ which is defined globally on $M$.
One may then apply the result of Lempert on stability of CR
functions to $G$. The stability of the pluriharmonic functions
follows by consideration of the real part.
\end{proof}

The proof of Theorem \ref{main} follows because under the hypothesis
of the theorem that the kernel of $P_4$ is exactly the CR
pluri-harmonic functions, we obtain via Proposition \ref{pluriext}
that the Paneitz operator $P_4$ is stable. Thus the hypothesis of
Proposition \ref{keyproppan} is satisfied and we may conclude that
$P_4^t\geq 0$ for $|t|<\delta$.

Deformation functions $\psi(\cdot,t)=t\phi(\cdot)$ where $\phi\in
C^\infty(S^3)$ have been studied in an important paper by
Burns-Epstein\cite{BC}. The standard CR structure on $S^3$ has
vanishing torsion. Thus combining the results in \cite {BC} and our
Theorem \ref{main} we also have:

\begin{cor}
Suppose that $(S^{3},J_{0})$ is the sphere $S^3$ equipped with the
standard CR structure and $\psi(\cdot,t)=t\phi(\cdot)$ is a
deformation function where $\phi(\cdot)$ satisfies the Burns-Epstein
condition. If we define the deformation $J_t$ of the CR structures
by $\psi(\cdot,t)$ then $P_{4}^{t}\geq 0$ for $t$ small enough.
\end{cor}

 The previous Corollary when combined with the results in \cite{B},
 \cite{BC} and \cite{CCY}, easily yields Theorem \ref{deformationthm}
 of the introduction.

\section{The Webster curvature for Ellipsoids}\label{Webcur}
In this section, we are going to show Theorem \ref{main2} of the
introduction. We will need a formula for the Webster curvature for
hypersurfaces embedded in $\mathbb{C}^2$ in a form suitable for our
computations. Other formulae have been derived in \cite{LL}, see
Theorem 1.1 there.

Let $M\hookrightarrow \mathbb{C}^{2}$ be a hypersurface defined by a
defining function $u(z_{1},z_{2})$:
\[M^{3}=\{(z_{1},z_{2})\in \mathbb{C}^{2}\ |\  u(z_{1},z_{2})=0\},\]
where $du(z)\neq 0$ for all $z\in M$. Equipped with the induced CR
structure from $\mathbb{C}^{2}$ and the contact form
\[\theta=\frac{i(\bar{\partial}u-\partial{u})}{2}|_{M^{3}},\]
$M$ is a pseudohermitian manifold, provided that $\theta\wedge
d\theta\neq 0$. It is easy to see that the induced CR structure can
be defined by the complex $(1,0)$-vector
\begin{equation}
Z_{1}=u_{2}\frac{\partial }{\partial z_{1}}-u_{1}\frac{\partial
}{\partial z_{2}}.
\end{equation}
We will use the notations:
\[u_{j}=\frac{\partial u}{\partial z_{j}},\ \ \ u_{jk}=\frac{\partial^{2} u}{\partial z_{j}\partial z_{k}},\]
for all $j, k\in\{1,2, \bar{1}, \bar{2}\}$. The characteristic
vector field $T$ is a real vector field which is uniquely defined by
\begin{equation}
d\theta(T\wedge\cdot)=0,\ \ \ \theta(T)=1.
\end{equation}

Let $\{\theta^{1}, \theta^{\bar 1}, \theta\}$ be the dual frame to
$\{Z_{1}, Z_{\bar 1}, T\}$. Then we have
\begin{equation}
d\theta=ih_{1\bar 1}\theta^{1}\wedge\theta^{\bar 1},
\end{equation}
for some nonzero real function $h_{1\bar 1}$. If necessary, we could
change the sign for $u$ and assume, without loss of generality, that
$h_{1\bar 1}>0$.

Let
\begin{equation}
J(u)=\left|\begin{array}{ccc}
u&u_{\bar 1}&u_{\bar 2}\\
u_{1}&u_{1\bar 1}&u_{1\bar 2}\\
u_{2}&u_{2\bar 1}&u_{2\bar 2}.
\end{array}\right|
\end{equation}

\begin{prop}\label{fefferman}
On $M^{3}$, we have
\begin{equation}\label{levi}
h_{1\bar 1}=-J(u).
\end{equation}
\end{prop}
\begin{proof}We compute, on $M$,
\begin{equation}
\begin{split}
d\theta&=-id(\partial u)\\
&=-id(\sum_{j=1}^{2}u_{j}dz_{j})=-i(\sum_{j=1}^{2}du_{j}\wedge dz_{j})\\
&=i\sum_{j,k=1}^{2}u_{jk}dz_{j}\wedge dz_{k}+i\sum_{j,k=1}^{2}u_{j\bar k}dz_{j}\wedge dz_{\bar k}\\
&=i\sum_{j,k=1}^{2}u_{j\bar k}dz_{j}\wedge dz_{\bar k}.
\end{split}
\end{equation}
Therefore
\begin{equation}
\begin{split}
h_{1\bar 1}&=-id\theta(Z_{1}\wedge Z_{\bar 1})\\
&=\sum_{j,k=1}^{2}u_{j\bar k}dz_{j}\wedge dz_{\bar k}((u_{2}\frac{\partial}{\partial z_{1}}-u_{1}\frac{\partial}{\partial z_{2}})\wedge(u_{\bar 2}\frac{\partial}{\partial z_{\bar 1}}-u_{\bar 1}\frac{\partial}{\partial z_{\bar 2}}))\\
&=u_{1\bar 1}u_{2}u_{\bar 2}+u_{2\bar 2}u_{1}u_{\bar 1}-u_{1\bar 2}u_{2}u_{\bar 1}-u_{2\bar 1}u_{1}u_{\bar 2}\\
&=-\left|
\begin{array}{ccc}
0&u_{\bar 1}&u_{\bar 2}\\
u_{1}&u_{1\bar 1}&u_{1\bar 2}\\
u_{2}&u_{2\bar 1}&u_{2\bar 2}
\end{array}
\right| =-\left|
\begin{array}{ccc}
u&u_{\bar 1}&u_{\bar 2}\\
u_{1}&u_{1\bar 1}&u_{1\bar 2}\\
u_{2}&u_{2\bar 1}&u_{2\bar 2}
\end{array}
\right| =-J(u), \ \ \ \ \textrm{on}\ \ M.
\end{split}
\end{equation}
\end{proof}

Let
\begin{equation}
U=(U_{ab})_{3\times 3}=\left[
\begin{array}{ccc}
u&u_{\bar 1}&u_{\bar 2}\\
u_{1}&u_{1\bar 1}&u_{1\bar 2}\\
u_{2}&u_{2\bar 1}&u_{2\bar 2}
\end{array}
\right].
\end{equation}
That is, $U_{ba}=\overline{U_{ab}}$, and
\begin{equation}
\begin{split}
&U_{11}=u;\ U_{12}=u_{\bar 1};\ U_{13}=u_{\bar 2};\\
&U_{(j+1)(k+1)}=u_{j\bar k},\ \ \ \ 1\leq j,k\leq 2.
\end{split}
\end{equation}

Note that $h_{1\bar 1}>0$, so the matrix $U$ is invertible on a
neighborhood of $M$. Let $U^{-1}=(U^{ab})$ be the inverse of $U$.
Then it is easy to show that $U^{ba}=\overline{U^{ab}}$ and
\begin{equation}
\begin{split}
U^{11}&=\frac{u_{1\bar 2}u_{2\bar 1}-u_{1\bar 1}u_{2\bar
2}}{h_{1\bar 1}};\ U^{12}=\frac{u_{2\bar 2}u_{\bar 1}-u_{2\bar
1}u_{\bar 2}}{h_{1\bar 1}};\
U^{13}=\frac{u_{1\bar 1}u_{\bar 2}-u_{1\bar 2}u_{\bar 1}}{h_{1\bar 1}};\\
U^{22}&=\frac{-u_{2}u_{\bar 2}}{h_{1\bar 1}};\
U^{23}=\frac{-u_{1}u_{\bar 2}}{h_{1\bar 1}};\
U^{33}=\frac{u_{1}u_{\bar 1}}{h_{1\bar 1}}.
\end{split}
\end{equation}

\begin{prop}
On $M$,
\begin{equation}
\begin{split}
T&=\sum_{j=1}^{2}iU^{1(j+1)}\frac{\partial}{\partial z_{j}}+\ \textrm{complex conjugate}\\
\theta^{1}&=U^{13}dz_{1}-U^{12}dz_{2}.
\end{split}
\end{equation}
\end{prop}
\begin{proof}
We just check that $T$ satisfies
\[d\theta(T\wedge\cdot)=0,\ \ \ \theta(T)=1.\]
We compute
\begin{equation}
\begin{split}
d\theta(T\wedge\cdot)&=i(\sum_{j,k=1}^{2}u_{j\bar k}dz_{j}\wedge dz_{\bar k})(T\wedge\cdot)\\
&=i\sum_{j,k=1}^{2}u_{j\bar k}(dz_{j}(T)dz_{\bar k}-dz_{\bar k}(T)dz_{j})\\
&=-\sum_{j,k=1}^{2}(U_{(j+1)(k+1)}U^{1(j+1)}dz_{\bar k}+U_{(j+1)(k+1)}U^{(k+1)1}dz_{j})\\
&=-\sum_{k=1}^{2}(\delta_{k3}-U_{1(k+1)}U^{11} )dz_{\bar k}-\sum_{j=1}^{2}(\delta_{3j}-U^{11}U_{(j+1)1} )dz_{j}\\
&=\sum_{j=1}^{2}U_{1(j+1)}U^{11}dz_{\bar j}+U^{11}U_{(j+1)1}dz_{j}\\
&=U^{11}(\sum_{j=1}^{2}u_{j}dz_{j}+u_{\bar j}dz_{\bar j})\\
&=U^{11}du=0,\ \ \ \textrm{on}\ \ M,
\end{split}
\end{equation}
and
\begin{equation}
\begin{split}
\theta(T)&=-i(\sum_{j=1}^{2}u_{j}dz_{j})(T)\\
&=\sum_{j=1}^{2}u_{j}U^{1(j+1)}=\sum_{j=1}^{2}U_{(j+1)1}U^{1(j+1)}\\
&=\sum_{b=1}^{3}U_{b1}U^{1b},\ \ \textrm{on}\ M\ (U_{11}=0,\ \ \textrm{on}\ M)\\
&=1
\end{split}
\end{equation}
Similarly, after a direct computation, we get $\theta^{1}(T)=0$ and
$\theta^{1}(Z_{1})=1$.
\end{proof}

\begin{prop}
With respect to the frame $Z_{1}$, the connection form
$\theta_{1}{}^{1}$ and the torsion form $\tau^{1}$ are expressed by
\begin{equation}\label{connetor}
\begin{split}
\theta_{1}{}^{1}&=(h^{1\bar 1}Z_{1}h_{1\bar 1})\theta^{1}+((u_{1}Tc_{2}-u_{2}Tc_{1})+i(c_{1}Z_{1}c_{2}-c_{2}Z_{1}c_{1}))\theta,\\
\tau^{1}&=-i(c_{1}Z_{\bar 1}c_{2}-c_{2}Z_{\bar 1}c_{1})\theta^{\bar
1},
\end{split}
\end{equation}
where $h^{1\bar 1}=h_{1\bar 1}^{-1},\ c_{1}=U^{13}$ and
$c_{2}=-U^{12}$.
\end{prop}
\begin{proof}
First we point out that all equalities are only true on $M$. Now let
$\theta^{1}=c_{1}dz_{1}+c_{2}dz_{2}$, i.e., $c_{1}=U^{13}$ and
$c_{2}=-U^{12}$. We have that
$1=\theta^{1}(Z_{1})=c_{1}u_{2}-c_{2}u_{1}$. Therefore
\begin{equation}
\begin{split}
u_{2}\theta^{1}&=u_{2} c_{1}dz_{1}+u_{2} c_{2}dz_{2}\\
&=dz_{1}+c_{2}u_{1}dz_{1}+u_{2} c_{2}dz_{2}\\
&=dz_{1}+c_{2}(u_{1}dz_{1}+u_{2}dz_{2})\\
&=dz_{1}+ic_{2}\theta,
\end{split}
\end{equation}
or
\begin{equation}\label{3}
dz_{1}=u_{2}\theta^{1}-ic_{2}\theta,
\end{equation}
hence,
\begin{equation}
\begin{split}
0&=d(dz_{1})=d(u_{2}\theta^{1}-ic_{2}\theta)\\
&=du_{2}\wedge\theta^{1}+u_{2}d\theta^{1}-idc_{2}\wedge\theta-ic_{2}d\theta,
\end{split}
\end{equation}
or
\begin{equation}\label{10}
u_{2}d\theta^{1}=-du_{2}\wedge\theta^{1}+idc_{2}\wedge\theta+ic_{2}d\theta.
\end{equation}
Similarly, we have
\begin{equation}\label{4}
dz_{2}=-u_{1}\theta^{1}+ic_{1}\theta,
\end{equation}
and thus,
\begin{equation}\label{11}
u_{1}d\theta^{1}=-du_{1}\wedge\theta^{1}+idc_{1}\wedge\theta+ic_{1}d\theta.
\end{equation}
Taking together (\ref{10}) and (\ref{11}), one obtains that
\begin{equation}\label{st1}
\begin{split}
d\theta^{1}&=(c_{1}u_{2}-c_{2}u_{1})d\theta^{1}\\
&=\theta^{1}\wedge(c_{1}du_{2}-c_{2}du_{1})+\theta\wedge
i(c_{2}dc_{1}-c_{1}dc_{2})
\end{split}
\end{equation}
On the other hand, we have
\begin{equation}\label{st2}
d\theta^{1}=\theta^{1}\wedge\theta_{1}{}^{1}+\theta\wedge\tau^{1}.
\end{equation}
From (\ref{st1}), (\ref{st2}) and by the Cartan lemma, there exists
functions $a, b$ and $c$ such that
\begin{equation}\label{st3}
\begin{split}
\theta_{1}{}^{1}&=c_{1}du_{2}-c_{2}du_{1}+a\theta^{1}+b\theta\\
\tau^{1}&=i(c_{2}dc_{1}-c_{1}dc_{2})+b\theta^{1}+c\theta.
\end{split}
\end{equation}
Since $\tau^{1}=A^{1}{}_{\bar 1}\theta^{\bar 1}$, from (\ref{st3}),
this means that
\begin{equation}
\begin{split}
A^{1}{}_{\bar 1}&=i(c_{2}Z_{\bar 1}c_{1}-c_{1}Z_{\bar 1}c_{2}),\\
b&=-i(c_{2}Z_{1}c_{1}-c_{1}Z_{1}c_{2}),\\
c&=-i(c_{2}Tc_{1}-c_{1}Tc_{2}),
\end{split}
\end{equation}
hence,
\begin{equation}
\theta_{1}{}^{1}=c_{1}du_{2}-c_{2}du_{1}+a\theta^{1}-i(c_{2}Z_{1}c_{1}-c_{1}Z_{1}c_{2})\theta.
\end{equation}
Finally, from the structural equation $h^{1\bar 1}dh_{1\bar
1}=\theta_{1}{}^{1}+\theta_{\bar 1}{}^{\bar 1}$, we get
\begin{equation}
a=c_{2}Z_{1}u_{1}-c_{1}Z_{1}u_{2}+h^{1\bar 1}Z_{1}h_{1\bar 1},
\end{equation}
hence,
\begin{equation}
\begin{split}
\theta_{1}{}^{1}&=c_{1}du_{2}-c_{2}du_{1}+(c_{2}Z_{1}u_{1}-c_{1}Z_{1}u_{2}+h^{1\bar 1}Z_{1}h_{1\bar 1})\theta^{1}-i(c_{2}Z_{1}c_{1}-c_{1}Z_{1}c_{2})\theta\\
&=(h^{1\bar 1}Z_{1}h_{1\bar
1})\theta^{1}+((u_{1}Tc_{2}-u_{2}Tc_{1})+i(c_{1}Z_{1}c_{2}-c_{2}Z_{1}c_{1}))\theta,
\end{split}
\end{equation}
where in the last equality, we used the identities $u_{1}Z_{\bar
1}c_{2}-u_{2}Z_{\bar 1}c_{1}=0$ and $c_{1}u_{2}-c_{2}u_{1}=1$.
\end{proof}

\begin{prop}
The Webster curvature can be expressed as
\begin{equation}\label{exofWebster}
R=-h^{1\bar 1}(Z_{\bar 1}Z_{1}\log{h_{1\bar
1}})+i\Big((u_{1}Tc_{2}-u_{2}Tc_{1})+i(c_{1}Z_{1}c_{2}-c_{2}Z_{1}c_{1})\Big).
\end{equation}
\end{prop}
\begin{proof}
Let
$E=(u_{1}Tc_{2}-u_{2}Tc_{1})+i(c_{1}Z_{1}c_{2}-c_{2}Z_{1}c_{1})$.
Taking the exterior differential of $\theta_{1}{}^{1}$
\begin{equation}\label{dofconn}
\begin{split}
d\theta_{1}{}^{1}&=d(Z_{1}\log{h_{1\bar 1}})\wedge\theta^{1}+(Z_{1}\log{h_{1\bar 1}})d\theta^{1}+dE\wedge\theta+Ed\theta\\
&=(-Z_{\bar 1}Z_{1}\log{h_{1\bar 1}}+\theta_{1}{}^{1}(Z_{\bar 1})(Z_{1}\log{h_{1\bar 1}})+ih_{1\bar 1}E)\theta^{1}\wedge\theta^{\bar 1},\ \textrm{mod}\ \theta\\
&=(-Z_{\bar 1}Z_{1}\log{h_{1\bar 1}}+ih_{1\bar
1}E)\theta^{1}\wedge\theta^{\bar 1},\ \textrm{mod}\ \theta.
\end{split}
\end{equation}
Comparing (\ref{dofconn}) with the structure equation
$d\theta_{1}{}^{1}=h_{1\bar 1}R\theta^{1}\wedge\theta^{\bar 1},\
\textrm{mod}\ \theta$, we immediately get formula
(\ref{exofWebster}) for the Webster curvature.
\end{proof}
Finally combining (\ref{connetor}) and (\ref{exofWebster}), we get
another representation for the connection form
\begin{equation}\label{conn}
\theta_{1}{}^{1}=(Z_{1}\log{h_{1\bar
1}})\theta^{1}-i\left(R+h^{1\bar 1}(Z_{\bar 1}Z_{1}\log{h_{1\bar
1}})\right)\theta.
\end{equation}
\begin{rem}
There is another expression for the Webster curvature, which was
proved by S.-Y. Li  and H.-S. Luk in \cite{LL}. It is
\begin{equation}\label{LLweb}
R=-h^{1\bar
1}\sum_{j,k=1}^{2}\frac{\partial^{2}\log{(-J(u))}}{\partial
z_{j}\partial z_{\bar k}}w^{j}w^{\bar k}+2\frac{\det{H(u)}}{h_{1\bar
1}}.
\end{equation}
\end{rem}

Next an ellipsoid is given by
\begin{equation}
A_{1}x_{1}{}^{2}+B_{1}y_{1}{}^{2}+A_{2}x_{2}{}^{2}+B_{2}y_{2}{}^{2}-1=0,
\end{equation}
where $A_{1}, A_{2}, B_{1}, B_{2}>0$. Set $z_{1}=x_{1}+iy_{1}$ and $z_{2}=x_{2}+iy_{2}$. Then our defining function becomes
\begin{equation}
u=b_{1}|z_{1}|^{2}+b_{2}|z_{2}|^{2}+a_{1}z_{1}{}^{2}+a_{1}\bar{z}_{1}{}^{2}+a_{2}z_{2}{}^{2}+a_{2}\bar{z}_{2}{}^{2}-1=0,
\end{equation}
where $a_{j}=\frac{1}{4}(A_{j}-B_{j}),\
b_{j}=\frac{1}{2}(A_{j}+B_{j})>0,\ j=1,2.$ We want to make use of
the formula for Webster curvature (\ref{exofWebster})

$$R=-h^{1\bar 1}Z_{\bar 1}Z_{1}(\log{h_{1\bar
1}})+i\big[(u_{1}Tc_{2}-u_{2}Tc_{1})+i(c_{1}Z_{1}c_{2}-c_{2}Z_{1}c_{1})\big],$$
where we recall $Z_{1}=u_{2}\frac{\partial}{\partial
z_{1}}-u_{1}\frac{\partial}{\partial z_{2}}$ and
$\theta=\frac{1}{2i}(\partial u-\bar{\partial}u)$. The functions
$c_{1}, c_{2}$ satisfy the identities:
\begin{equation}
Z_{\bar 1}u_{1}=h_{1\bar 1}c_{1},\ \ \textrm{and}\ \ Z_{\bar
1}u_{2}=h_{1\bar 1}c_{2}.
\end{equation}
So,
\begin{equation}\label{fteq}
\begin{split}
Z_{1}Z_{\bar 1}u_{1}&=Z_{1}(h_{1\bar 1})c_{1}+h_{1\bar 1}Z_{1}c_{1}\\
Z_{1}Z_{\bar 1}u_{2}&=Z_{1}(h_{1\bar 1})c_{2}+h_{1\bar 1}Z_{1}c_{2}.
\end{split}
\end{equation}
Multiplying the first equation in (\ref{fteq}) by $c_{2}$ and the
second by $c_{1}$ and subtracting, we get
\begin{equation}
h_{1\bar 1}(c_{2}Z_{1}c_{1}-c_{1}Z_{1}c_{2})=c_{2}Z_{1}Z_{\bar
1}u_{1}-c_{1}Z_{1}Z_{\bar 1}u_{2}.
\end{equation}
So,
\begin{equation}\label{eqnforZ}
\begin{split}
c_{2}Z_{1}c_{1}-c_{1}Z_{1}c_{2}&=\frac{c_{2}Z_{1}Z_{\bar 1}u_{1}-c_{1}Z_{1}Z_{\bar 1}u_{2}}{h_{1\bar 1}}\\
&=\frac{c_{2}([Z_{1},Z_{\bar 1}]u_{1})-c_{1}([Z_{1},Z_{\bar
1}]u_{2})}{h_{1\bar 1}}+
\frac{c_{2}Z_{\bar 1}Z_{1}u_{1}-c_{1}Z_{\bar 1}Z_{1}u_{2}}{h_{1\bar 1}}\\
&=\frac{-ih_{1\bar 1}(c_{2}Tu_{1}-c_{1}Tu_{2})}+\frac{c_{2}Z_{\bar 1}Z_{1}u_{1}-c_{1}Z_{\bar 1}Z_{1}u_{2}}{h_{1\bar 1}}\\
&=i(c_{1}Tu_{2}-c_{2}Tu_{1})+\frac{c_{2}Z_{\bar
1}Z_{1}u_{1}-c_{1}Z_{\bar 1}Z_{1}u_{2}}{h_{1\bar 1}}.
\end{split}
\end{equation}
Next note $\theta^1(Z_{1})=c_{1}u_{2}-u_{1}c_{2}=1$. Thus,
$T(c_{1}u_{2}-u_{1}c_{2})=0$. So we have,
\begin{equation}\label{eqnforc1}
u_{2}Tc_{1}-u_{1}Tc_{2}=c_{2}Tu_{1}-c_{1}Tu_{2}.
\end{equation}
From (\ref{eqnforZ}) and (\ref{eqnforc1}) we get,
\begin{equation}\label{fteq2}
\begin{split}
&(u_{1}Tc_{2}-u_{2}Tc_{1})+i(c_{1}Z_{1}c_{2}-c_{2}Z_{1}c_{1})\\
&=(c_{1}Tu_{2}-c_{2}Tu_{1})+i(-i)(c_{1}Tu_{2}-c_{2}Tu_{1})-i\left(\frac{c_{2}Z_{\bar 1}Z_{1}u_{1}-c_{1}Z_{\bar 1}Z_{1}u_{2}}{h_{1\bar 1}}\right)\\
&=2(c_{1}Tu_{2}-c_{2}Tu_{1})-i\left(\frac{c_{2}Z_{\bar
1}Z_{1}u_{1}-c_{1}Z_{\bar 1}Z_{1}u_{2}}{h_{1\bar 1}}\right).
\end{split}
\end{equation}
Substituting ($\ref{fteq2}$) into the Webster curvature formula
(\ref{exofWebster}), we get
\begin{equation}\label{exofWeb1}
\begin{split}
R&=-h^{1\bar 1}Z_{\bar 1}Z_{1}(\log{h_{1\bar 1}})+i\big[(u_{1}Tc_{2}-u_{2}Tc_{1})+i(c_{1}Z_{1}c_{2}-c_{2}Z_{1}c_{1})\big]\\
&=-h^{1\bar 1}Z_{\bar 1}Z_{1}(\log{h_{1\bar
1}})+2i(c_{1}Tu_{2}-c_{2}Tu_{1})+\left(\frac{c_{2}Z_{\bar
1}Z_{1}u_{1}-c_{1}Z_{\bar 1}Z_{1}u_{2}}{h_{1\bar 1}}\right).
\end{split}
\end{equation}
We next compute an expression for the Levi form. We have,
\begin{equation}
\begin{split}
J(u)&=\left|\begin{array}{ccc}
u&u_{\bar 1}&u_{\bar 2}\\
u_{1}&u_{1\bar 1}&u_{1\bar 2}\\
u_{2}&u_{\bar{1}2}&u_{2\bar 2}
\end{array}\right|=
\left|\begin{array}{ccc}
u&u_{\bar 1}&u_{\bar 2}\\
u_{1}&b_{1}&0\\
u_{2}&0&b_{2}
\end{array}\right|\\
&=\left|\begin{array}{ccc}
u&2a_{1}\bar{z}_{1}+b_{1}z_{1}&2a_{2}\bar{z}_{2}+b_{2}z_{2}\\
2a_{1}z_{1}+b_{1}\bar{z}_{1}&b_{1}&0\\
2a_{2}z_{2}+b_{2}\bar{z}_{2}&0&b_{2}
\end{array}\right|\\
&=b_{1}b_{2}u-b_{1}u_{2}u_{\bar 2}-b_{2}u_{1}u_{\bar 1}.
\end{split}
\end{equation}\label{exoflevi}
Thus when $u=0$, one has
\begin{equation}
h_{1\bar 1}=b_{1}u_{2}u_{\bar 2}+b_{2}u_{1}u_{\bar 1}.
\end{equation}
Next we compute the first term in the Webster curvature formula (\ref{exofWeb1}):
\begin{equation}\label{fte4}
\begin{split}
-h^{1\bar 1}Z_{\bar 1}Z_{1}(\log{h_{1\bar 1}})&=-h^{1\bar 1}Z_{\bar 1}\left(\frac{Z_{1}h_{1\bar 1}}{h_{1\bar 1}}\right)\\
&=\frac{|Z_{1}h_{1\bar 1}|^2}{(h_{1\bar 1})^3}-\frac{Z_{\bar
1}Z_{1}h_{1\bar 1}}{(h_{1\bar 1})^2},\ \textrm{using}\ h^{1\bar
1}=\frac{1}{h_{1\bar 1}}.
\end{split}
\end{equation}
A straightforward computation using the defining function yields,
\begin{equation}
\begin{split}
Z_{\bar 1}u_{2}&=-b_{2}u_{\bar 1},\ Z_{\bar 1}u_{1}=b_{1}u_{\bar 2}\\
Z_{\bar 1}u_{\bar 1}&=2a_{1}u_{\bar 2},\ Z_{\bar 1}u_{\bar
2}=-2a_{2}u_{\bar 1}.
\end{split}
\end{equation}
So,
\begin{equation}\label{fte5}
\begin{split}
Z_{\bar 1}Z_{1}h_{1\bar 1}&=2a_{1}b_{2}Z_{\bar 1}(u_{2}u_{\bar 1})-2a_{2}b_{1}Z_{\bar 1}(u_{1}u_{\bar 2})\\
&=2a_{1}b_{2}[(Z_{\bar 1}u_{2})u_{\bar 1}+u_{2}(Z_{\bar 1}u_{\bar 1})]-2a_{2}b_{1}[(Z_{\bar 1}u_{1})u_{\bar 2}+u_{1}(Z_{\bar 1}u_{\bar 2})]\\
&=4a_{1}{}^{2}b_{2}|u_{2}|^{2}+4a_{2}{}^{2}b_{1}|u_{1}|^{2}-2a_{1}b_{2}{}^{2}(u_{\bar
1})^{2}-2a_{2}b_{1}{}^{2}(u_{\bar 2})^{2}.
\end{split}
\end{equation}
Next we claim that
\begin{equation}\label{claim1}
2i(c_{1}Tu_{2}-c_{2}Tu_{1})=\frac{2}{h_{1\bar
1}^{2}}\left(b_{1}b_{2}^{2}|u_{1}|^{2}+b_{2}b_{1}^{2}|u_{2}|^{2}-
2a_{1}b_{2}^{2}|u_{\bar 1}|^{2}-2a_{2}b_{1}^{2}|u_{\bar
1}|^{2}\right);
\end{equation}
and
\begin{equation}\label{claim2}
\frac{c_{2}Z_{\bar 1}Z_{1}u_{1}-c_{1}Z_{\bar 1}Z_{1}u_{2}}{h_{1\bar
1}}=\frac{2a_{1}b_{2}^{2}|u_{\bar 1}|^{2}+2a_{2}b_{1}^{2}|u_{\bar
1}|^{2}}{(h_{1\bar 1})^{2}}
\end{equation}
These follow because,
\begin{equation}
Z_{\bar 1}Z_{1}u_{1}=Z_{\bar 1}(2a_{1}u_{2})=-2a_{1}b_{2}u_{\bar 1}.
\end{equation}
So,
\begin{equation}
c_{2}Z_{\bar 1}Z_{1}u_{1}=\frac{-b_{2}u_{\bar 1}}{h_{1\bar
1}}(-2a_{1}b_{2}u_{\bar 1})=\frac{2a_{1}b_{2}^{2}(u_{\bar
1})^{2}}{h_{1\bar 1}}.
\end{equation}
Similarly,
\begin{equation}
c_{1}Z_{\bar 1}Z_{1}u_{2}=\frac{b_{1}u_{\bar 2}}{h_{1\bar
1}}(-2a_{2}b_{1}u_{\bar 2})=\frac{-2a_{2}b_{1}^{2}(u_{\bar
2})^{2}}{h_{1\bar 1}}.
\end{equation}
Taking the two expressions above together, we get the second claim
(\ref{claim2}). To prove the first claim (\ref{claim1}), we use
\begin{equation}
T=\frac{i}{h_{1\bar 1}}\left(b_{2}u_{\bar 1}\frac{\partial}{\partial
z_{1}}+b_{1}u_{\bar 2}\frac{\partial}{\partial z_{2}}\right)+\
\textrm{complex conjugate}.
\end{equation}
So
\begin{equation}
\begin{split}
Tu_{1}&=\frac{2ia_{1}b_{2}u_{\bar 1}}{h_{1\bar 1}}-\frac{ib_{1}b_{2}u_{1}}{h_{1\bar 1}}\\
Tu_{2}&=\frac{2ia_{2}b_{1}u_{\bar 2}}{h_{1\bar
1}}-\frac{ib_{1}b_{2}u_{2}}{h_{1\bar 1}},
\end{split}
\end{equation}
In conjunction with $c_{1}=\frac{b_{1}u_{\bar 2}}{h_{1\bar 1}},\
c_{2}=\frac{-b_{2}u_{\bar 1}}{h_{1\bar 1}}$, we get the first claim
(\ref{claim1}). Now we substitute (\ref{fte5}) into (\ref{fte4}) and
substitute (\ref{fte4}), (\ref{claim1}) and (\ref{claim2}) into
formula (\ref{exofWeb1}), to obtain
\begin{equation}
\begin{split}
R&=\frac{|Z_{1}h_{1\bar 1}|^2}{(h_{1\bar 1})^3}+\frac{2b_{1}(b_{2}^{2}-2a_{2}^{2})|u_{1}|^{2}+2b_{2}(b_{1}^{2}-2a_{1}^{2})|u_{2}|^{2}}{(h_{1\bar 1})^2}\\
&>0,
\end{split}
\end{equation}
The last inequality is a consequence of  $b_{i}^{2}-2a_{i}^{2}>0,\
i=1,2$. This follows because,
\begin{equation}
\begin{split}
b_{i}^{2}&=\frac{1}{4}(A_{i}+B_{i})^{2}\\
2a_{i}^{2}&=\frac{2}{16}(A_{i}-B_{i})^{2},
\end{split}
\end{equation}
hence $b_{i}^{2}-2a_{i}^{2}=\frac{1}{8}(A_{i}^{2}+B_{i}^{2}+6A_{i}B_{i})>0$, (note that $A_{i}, B_{i}>0$).

\section{Further Remarks}

It remains an interesting problem to determine the precise
geometrical condition when the notion of being in the kernel of
$P_4$ coincides with CR-pluriharmonicity for a general CR structure.
One problem of immediate interest is to determine if for embedded
structures, the CR-pluriharmonic functions coincide with the
functions in the kernel of the Paneitz operator. It is unclear if
such an equivalence is true even for CR structures close to the
standard structure on $S^3$. Under the assumption that the CR
pluriharmonic functions coincide with the kernel of the Paneitz
operator, C. R. Graham, K. Hirachi and J. M. Lee proved the theorem
stated below. Thus our question has further geometric consequences
beyond a possible link with embedding of CR structures.
\begin{thm}\label{main5}
Let $\Omega\subset\C^{2}$ be a strictly pseudoconvex domain with a defining function $u$. Suppose $M=\partial\Omega$. Then the
following are equivalent:
\begin{enumerate}
   \item $Q=0.$
   \item $u$ satisfies Fefferman's Monge-Ampere equation $-J(u)\equiv 1$ along $M$
up to multiplication by a CR pluriharmonic function.
   \item $\theta\wedge d\theta$ is the invariant volume
element up to multiplication by a CR pluriharmonic function.
\end{enumerate}
\end{thm}
We recall that $J(u)$ is Fefferman's Monge-Ampere equation, which is
defined by

\begin{equation*}
J(u)=\det\left(\begin{array}{ccc}
{u}_{1\bar 1}&{u}_{1\bar 2}&{u}_{1}\\
{u}_{2\bar 1}&{u}_{2\bar 2}&{u}_{2}\\
{u}_{\bar 1}&{u}_{\bar 2}&{u}
\end{array}
 \right)
\end{equation*}

In section \ref{Webcur}, we showed that the Webster curvature for
ellipsoids are positive. It is interesting to know if the Webster
curvature is also positive for a strictly convex domain? If so, then
from our earlier result in \cite{CCY}, there is an uniform positive
lower bound for the first nonzero eigenvalues $\lambda_{t}$ of the
Kohn Laplacian $\Box_{b}^{t}$ for the family of strictly convex
domains $\Omega_{t}$, which is smoothly dependent on $t$.

\bibliographystyle{plain}

\end{document}